\documentclass{elsarticle}

\usepackage{geometry}
\usepackage{amsthm}
\usepackage{amssymb}
\usepackage{amsmath}
\usepackage{enumitem}
\usepackage[british]{babel}
\usepackage[mathscr]{euscript}

\newtheorem{tm}{Theorem}[section]

\newtheorem{prop}[tm]{Proposition}
\newtheorem{lem}[tm]{Lemma}

\theoremstyle{definition}
\newtheorem{ex}[tm]{Example}

\theoremstyle{remark}

\makeatletter\@namedef{subjclassname@2020}{\textup{2020} Mathematics Subject Classification}\makeatother

\begin{document}

\newgeometry{textheight = 22.5 cm, textwidth=13.5 cm, centering}

\makeatletter
\def\ps@pprintTitle{%
  \let\@oddhead\@empty
  \let\@evenhead\@empty
  \let\@oddfoot\@empty
  \let\@evenfoot\@oddfoot
}
\makeatother

\title{Virtual Lie subgroups of locally compact groups}
\author{Antoni Machowski\footnote{Postal address: Salwatorska Street 23/7, 30-117 Kraków, Poland; e-mail address: antoni.machowski@doctoral.uj.edu.pl; phone number: +48 694 460 464. Declarations of interest: none.}}
\begin{abstract}
We examine subgroups of locally compact groups that are continuous homomorphic images of connected Lie groups and we give a criterion for being such an image. We also provide a new characterisation of Lie groups and a characterisation of groups that are images of connected locally compact groups.
\end{abstract}
\maketitle
\begin{center}
Faculty of Mathematics and Computer Science, Jagiellonian University, Stanis\l{}awa \L{}ojasiewicza Street 6, 30-348 Krak\'ow, Poland
\end{center}

\section{Outline of the problem}

By a \textit{topological group} we mean a group $H$ equipped with a Hausdorff topology $\tau$ such that the group operation $H\times H \ni(h_1,h_2)\mapsto h_1h_2\in H$ and the inverse mapping $H\ni h\mapsto h^{-1}\in H$ are continuous with respect to $\tau$. By a \textit{Lie group} we mean a topological group $S$ with a topology of a smooth manifold such that the group operation and the inverse mapping are smooth.

Let $S$ be a Lie group and let $H$ be its arbitrary subgroup. We call $H$ a \textit{virtual Lie subgroup} of $S$ if it admits a topology stronger than the induced one which makes it a connected Lie group. The Yamabe Theorem gives a characterisation of such groups which goes as follows:

\begin{tm}[\cite{Yam50}, see also {\cite[Theorem 9.6.1]{HN12}}]
\label{tm:yam}
Let $S$ be a Lie group and let $H$ be its arbitrary subgroup. Then the following conditions are equivalent:
\begin{enumerate}[label=(\roman*)]
\item $H$ is a virtual Lie subgroup of $S$.
\item $H$ is arcwise-connected.
\end{enumerate}
\end{tm}

The Yamabe theorem classifies virtual Lie subgroups nicely by giving an equivalent condition of arcwise-connectivity which is an obvious necessary one. Gleason and Palais, while pursuing a general formula for a stronger topology making a group into a connected Lie group found another characterisation.

\begin{tm}[{\cite[Theorem 7.3]{GP57}}]
\label{tm:gp}
For an arcwise-connected topological group $(H,\tau)$ the following conditions are equivalent:
\begin{enumerate}[label=(\roman*)]
\item There exists a stronger topology on $H$ which makes it a connected Lie group.
\item The topology $\mathscr{M}(\tau)$ makes $H$ a Lie group (where $\mathscr{M}(\tau)$ is given by a base of all arcwise-connected components of all open subsets of $(H,\tau)$ and is always the weakest locally arcwise-connected topology on $H$ stronger than $\tau$).
\item The covering dimensions of all the locally connected continua of $H$ have a common upper bound.
\end{enumerate}
\end{tm}

It is worth noting that the above theorem neither needs local compactness of $H$, nor does it need $H$ to be embeddable into such a group. Unfortunately, the condition (iii) is not easy to verify. On the other hand, the history of solving the Hilbert's Fifth Problem showed that the NSS property is both powerful in its consequences as well as being relatively easy to establish. In this paper we utilise the NSS property along with several structural results of locally compact groups to find a result similar to Theorem \ref{tm:yam} but for subgroups of arbitrary locally compact groups. Theorem \ref{tm:NtoL} states that an arbitrary subgroup of a locally compact group is a continuous homomorphic image of a connected Lie group provided it is arcwise connected and NSS. While arcwise connectivity is an obvious necessary condition, Example \ref{ex:omega} shows that it is not a sufficient one. In theorem \ref{tm:LtoN} we attempt to reverse Theorem \ref{tm:NtoL} but we cannot get the NSS property without an additional assumption which cannot be completely omitted as shown in Example \ref{ex:Q}. In the process of pursuing the generalisation of Teorem \ref{tm:yam}, we make a new characterisation of Lie groups in Theorem \ref{tm:Charlie}. We finish by examining groups that are continuous homomorphic images of connected locally compact groups in Theorem \ref{tm:Charloc} and Example \ref{ex:XY}.

\section{Main notions and preliminary results}

For a topological group $G$ we denote its identity element by $e$ or $e_G$ and by an $e$-\textit{neighbourhood} or $e_G$-\textit{neighbourhood} we mean a neighbourhood of this element in the topology of $G$.
For two topological groups $H$ and $S$ by denoting $H\simeq S$ we mean that the algebraic structures of $H$ and $S$ are isomorphic and by dentoting $H\equiv S$ we mean that $H$ and $S$ are isomorphic as topological groups. For a topological group $G$, by $Aut(G)$ we denote the group of automorphisms of $G$ equipped with the compact-open topology.
By a \textit{virtual Lie group} we mean a topological group that admits a stronger topology that makes it a connected Lie group. We will only consider this for groups that are embeddable in some locally compact group. We will call such groups \textit{pre-locally compact}.
We first make an easy observation that lets us study pre-locally compact groups without an explicit supergroup. We utilise the notion of Ra\u\i kov completion of a topological group (see \cite[Section 3.6]{AT08}).

\begin{prop}[cf. {\cite[Lemma 3.7.3]{AT08}}]
For a topological group $H$ the following conditions are equivalent:
\begin{enumerate}[label=(\roman*)]
\item $H$ is embeddable in some locally compact group.
\item The Ra\u\i kov completion of $H$ is locally compact.
\item There exists an $e_H$--neighbourhood $U$ such that for every $e_H$-neigh-bourhood $V$ there exist $h_1,h_2,\dots,h_n\in H$ such that $U\subseteq\cup_{i=1}^{n}h_i V$
\end{enumerate}
\end{prop}

We now make an easy observation, proof of which we leave to the reader. We look at $H$ as a subgroup of its Ra\u\i kov completion, but obviously the result would also be true if we took $H$ to be a subgroup of a locally compact group where $H$ is not necessarily dense.

\begin{prop}
\label{obs:phi}
Let $H$ be a pre-locally compact group and let $G$ be its Ra\u\i kov completion. Then the following conditions are equivalent:
\begin{enumerate}[label=(\roman*)]
\item $H$ admits a stronger topology making it into a connected Lie group.
\item There exist a connected Lie group $S$ and a continuous bijective group homomorphism $\varphi:S\to H$.
\item There exists a connected Lie group $S$ and a continuous group homomorphism $\varphi:S\to G$ such that $\varphi(S)=H$.
\end{enumerate}
Moreover, if these conditions are true, then $S$ and $\varphi$ in point (ii) are unique up to isomorphism.
\end{prop}

It is worth noting that the condition (iii) points to the fact that we are examining all continuous homomorphic images of connected Lie groups in locally compact groups, not only those that come from a one-to-one mapping. Now we will move to some serious auxiliary results which we will be using in the next section. We say that a topological group $H$ \textit{has no small subgroups} (or is an \textit{NSS group}) if it admits an $e$--neighbourhood $U$ which contains no non-trivial subgroups. This notion turned out to be a crucial one in finding a solution to the Hilbert's Fifth Problem which is stated below.

\begin{tm}[\cite{Yam53}, see also {\cite[Corollary 5.3.3]{Tao14}}]
\label{tm:hpv}
Let $G$ be a locally compact group. Then the following conditions are equivalent:
\begin{enumerate}[label=(\roman*)]
\item $G$ is a Lie group.
\item $G$ has a topology of a topological manifold.
\item $G$ is NSS.
\end{enumerate}
\end{tm}

Lastly, we envoke the celebrated Gleason--Yamabe Theorem on approximating connected locally compact groups by Lie groups.

\begin{tm}[\cite{Gle51}, \cite{Yam53}, see also {\cite[Theorem 1.1.13]{Tao14}}]
\label{tm:GY}
Let $G$ be a connected locally compact group. Then for any $e_G$--neighbourhood $U$ there exists a compact normal subgroup $K$ of $G$ such that $K\subseteq U$ and ${G}/{K}$ is a Lie group.
\end{tm}

\section{Main results}

We start this section by trying to establish sufficient conditions for being a virtual Lie group among pre-locally compact groups. Arcwise connectivity is an obvious necessary condition, but Example \ref{ex:omega} (see Section 4) shows that it is not sufficient on its own, contrary to the case of subgroups of Lie groups. However, in the case of Lie groups, one condition that is satisfied even by non-closed subgroups is the NSS condition. This motivates us to state the following.

\begin{tm}
\label{tm:NtoL}
Let $H$ be an arcwise-connected pre-locally compact NSS group. Then $H$ is a virtual Lie group.
\begin{proof}
Let $G$ be the Ra\u\i kov completion of $H$. Then $G$ is a connected locally compact group. Let $U_0$ be an $e_H$--neighbourhood without nontrivial subgroups. Let $U$ be an $e_G$--neighbourhood such that $U_0=U\cap H$. Then we get
\begin{equation}
\label{FcapH}
\forall{F\text{ subgroup of }G} : (F\subseteq U \implies F\cap H = \{e_G\}).
\end{equation}
Let $\mathscr{K}=\{K\lhd G \text{ compact}, K\subseteq U, G/K \text{ Lie}\}$. Then by Theorem \ref{tm:GY}, $\mathscr{K}$ is nonempty and $\bigcap \mathscr{K}=\{e_G\}$. Moreover $\mathscr{K}$ is closed under finite intersections. For $K\in \mathscr{K}$ let $\pi_K:G\to G/K$ be the cannonical projection. Then $\pi_K(H)$ is an arcwise connected subgroup of the connected Lie group $G/K$ and by Theorem \ref{tm:yam} there exists a topology $\tau_{K}$ on $\pi_K(H)$ making it into a connected Lie group.

To show that $H$ is a virtual Lie group let us fix $K\in\mathscr{K}$ again. Then by (\ref{FcapH}), $K\cap H= \{e_G\}$ so $\pi_{K}|_H:H\to G/K$ is injective and thus $\xi_{K}:=\pi_{K}|_H:H\to \pi_{K}|_H(H)$ is bijective. Hence $\xi_{K}$ induces a topology on $H$ by shifting the topology from $\pi_{K}(H)$ which makes $H$ into a connected Lie group. For any $K\in\mathscr{K}$ we will denote by $\tau(K)$ the original quotient topology on $\pi_K(H)$, by $\tau_K$ the topology of a connected Lie group on $\pi_K(H)$ and by $\tau_*^K$ the topology on $H$ given by shifting $\tau_K$ through $\xi_K$. Now we aim to prove:
\begin{equation}
\label{KinL}
K\subseteq L, K,L\in \mathscr{K} \implies \tau_*^K = \tau_*^L.
\end{equation}

First we will show that $\pi_K(H)\cap\pi_K(L)=\{e_{G/K}\}$. Indeed, if we have $\pi_K(g)\in\pi_K(H)$ for some $g\in L$, then $g\in H \cdot K$, so $g=h\cdot k$ for some $h\in H, k\in K$ and $L\ni gk^{-1}=h\in H$ which gives $\pi_K(g)=e_{G/K}$.
Now let $\rho:G/K\to G/L$ be the natural homomorphism. Then $\rho(\pi_K(H))=\pi_L(H)$ and $\rho|_{\pi_K(H)}:\pi_K(H)\to\pi_L(H)$ is a continous bijection because $\ker\rho=\pi_K(L)$ and thus $\ker\rho|_{\pi_K(H)}=\pi_K(L)\cap\pi_K(H)=\{e_{G/K}\}$.
Now $\xi_K:(H,\tau_*^K)\to (\pi_K(H),\tau_{K})$ and $\xi_L:(H,\tau_*^L)\to (\pi_L(H),\tau_{L})$ are isomorphisms of topological groups and $j_K: (\pi_K(H),\tau_K)\to (\pi_K(H),\tau(K))$ and  $\rho|_{\pi_K(H)}:(\pi_K(H),\tau(K))\to(\pi_L(H),\tau(L))$ are continuous bijective homomorphisms. Now $\rho\circ j_K \circ \xi_K:(H, \tau_*^K)\to (\pi_L(H),\tau(L))$ is a bijective continuous homomorphism from a connected Lie group to a virtual Lie group so by the uniqueness in Proposition \ref{obs:phi} we get that $\xi^{-1}_L\circ\rho\circ j_K\circ \xi_K:(H,\tau_*^K)\to(H,\tau_*^L)$ is an isomorphism of topological groups. Now we will show that this isomorphism is in fact the identity mapping. For $h\in H$, let $h':=\xi^{-1}_L\circ\rho\circ j_K\circ \xi_K(h)$. Then we get $\rho(\xi_K(h))=\xi_L(h')$ so $\rho(\pi_K(h))=\pi_L(h')$ so by the definition of $\rho$ we have $\pi_L(h)=\pi_L(h')$ but $\pi_L$ is injective so $h=h'$ which concludes the proof of (\ref{KinL}). It follows that for any $K,L\in \mathscr{K}$ we have $\tau_*^K=\tau_*^L$. Let us denote this topology by $\tau_*$. Then $(H,\tau_*)$ is a connected Lie group such that for every $K\in\mathscr{K}$ we get that $\pi_K|_H:(H,\tau_*)\to G/K$ is continuous. Now it is enough to show that $id:(H,\tau_*)\to H\subseteq G$ is continuous. Pick any $K\in\mathscr{K}$. Let $h_\sigma\in H$, $h_\sigma\to e_H$ w.r. to $\tau_*$. Then $\pi_K(h_\sigma)\to e_{G/K}$. Let $V$ be a compact $e_G$--neighbourhood. Since $\pi_K$ is open we have that $\pi_K(V)$ is a compact $e_{G/K}$--neighbourhood. Then there exists $\sigma_0$ such that for $\sigma\geq \sigma_0$ we have $\pi_K(h_\sigma)\in \pi_K(V)$ and then $h_\sigma\in K \cdot V$. Let $Z:=K\cdot V$. Then $Z$ is compact. Suppose that $h_\sigma$ does not converge to $e_G$. But $\{h_\sigma\}\subseteq Z$ so there is a subsequence $\Lambda$ such that $h_{\sigma(\lambda)}$ converges to $a\neq e_G$. Then we have $h_{\sigma(\lambda)}\to e_H$ w.r. to $\tau_*$ and $h_{\sigma(\lambda)}\to a\neq e_G$. Take any $L\in\mathscr{K}$. Then $\pi_L(h_{\sigma(\lambda)})$ converges to $e_{G/L}$ and to $\pi_L(a)$ in $G/L$ so $\pi_L(a)=e_{G/L}$ and $a\in \ker(\pi_L)=L$. Hence $a\in \bigcap \mathscr{K}=\{e_G\}$ which leads to a contradiction.
\end{proof}
\end{tm}

We now established that for pre-locally compact groups the NSS property and arcwise connectivity form a sufficient condition for being a virtual Lie group. Unfortunately, Example \ref{ex:Q} (see Section 4) shows that pre-locally compact virtual Lie groups need not be NSS. In the following theorem we prove that pre-locally compact virtual Lie groups are NSS if we add some additional conditions, which sadly cannot be dropped entirely.

\begin{tm}
\label{tm:LtoN}
Let $H$ be a pre-locally compact virtual Lie group such that $Z(H)$ is compact or torsion. Then 
\begin{enumerate}[label=(\arabic*)]
\item $H$ is arcwise connected,
\item there exists a compact group $K\lhd G$ such that $G/K$ is a Lie group and $H\cap K=\{e_G\}$ where $G$ is the Ra\u\i kov completion of $H$,
\item $H$ is NSS.
\end{enumerate}
In particular, $H$ is a normal subgroup of its Ra\u\i kov completion.
\begin{proof}
(1) is immediate from $H$ being a virtual Lie group regardless of the other conditions. (3) follows fom (2) since we get that $\pi_K|_H:H\to\pi_K(H)$ is an injective continuous homomorphism, so the NSS property for $\pi_K(H)$ implies the NSS property for $H$.
Now let $G:$ be the Ra\u\i kov completion of $H$. Then $G$ is a connected locally compact group. Let $\mathscr{K}$ be the family of compact normal subgroups of $G$ containted in $U$ and such that the quotient group is a Lie group. Since $H$ is a virtual Lie group, let $S$ be a connected Lie group and let $\varphi:S\to H$ be a continuous bijective homomorphism. Now let $\widetilde{K}$ be an element of $\mathscr{K}$ such that $\varphi^{-1}(\widetilde{K})$ is of smallest possible dimension. We will show that $\varphi^{-1}(\widetilde{K})$ is discrete. Suppose to the contrary that there is $x\in \varphi^{-1}(\widetilde{K})_0$ such that $x$ is not the identity of $S$. Then by the Gleason-Yamabe Theorem there exists $L\in \mathscr{K}$ such that $\varphi(x)\notin L$. Then $\widetilde{K}\cap L\in\mathscr{K}$ and $\varphi^{-1}(\widetilde{K}\cap L)$ is a closed subgroup of $\varphi^{-1}(\widetilde{K})$ which does not contain $\varphi^{-1}(\widetilde{K})_0$ so $\varphi^{-1}(\widetilde{K}\cap L)$ has smaller dimension than $\varphi^{-1}(\widetilde{K})$ which leads to a contradiction. Now $\varphi^{-1}(\widetilde{K})\lhd S$ is discrete and thus it is central. This means that $\varphi^{-1}(\widetilde{K})\equiv \mathbb{Z}^d \times F$ where $F$ is a finite abelian group (cf. {\cite[Corollary 8.A.23(2) (p. 183)]{CdlH16}}). Let now $F_*:=\{z\in\varphi^{-1}(\widetilde{K}):z\neq e_S, \text{ord}(z)<\infty\}$. Then $\varphi(F_*)$ is finite and $e_G\notin\varphi(F_*)$ so by Gleason-Yamabe Theorem there exists $K_*\in \mathscr{K}$ such that $K_*\cap\varphi(F_*)=\emptyset$. $K:=\widetilde{K}\cap K_*$ belongs to $\mathscr{K}$ and is torsion-free. Now $K\cap H=\varphi(\varphi^{-1}(K))\simeq \mathbb{Z}^d$ and it is a closed subgroup of $Z(H)$ from which it follows that $K\cap H=\{e_G\}$.
\end{proof}
\end{tm}

The following equivalence is an obvious consequence of Theorems \ref{tm:NtoL} and \ref{tm:LtoN}.

\begin{tm}
Let $H$ be a pre-locally compact group such that $Z(H)$ is compact or torsion. Then the following conditions are equivalent:
\begin{enumerate}[label=(\roman*)]
\item $H$ is a virtual Lie group.
\item $H$ is an arcwise-connected NSS group.
\end{enumerate} 
\end{tm}

Even though we were not able to provide an equivalent condition for being a virtual Lie group in full generality, our considerations enable us to find a new equivalent condition for being a Lie group.

\begin{tm}
\label{tm:Charlie}
Let $H$ be an arbitrary topological group. Then the following conditions are equivalent:
\begin{enumerate}[label=(\roman*)]
\item $H$ is a Lie group.
\item $H$ is a pre-locally compact locally arcwise-connected NSS group.
\end{enumerate}
\begin{proof}
All conditions in (ii) are easily satisfied by a Lie group. For the other implication we can assume that $H$ is connected. Since $H$ is also locally arcwise-connected, it is arcwise connected. This with the NSS condition by Theorem \ref{tm:NtoL} means that $H$ is a virtual Lie group. Now let $\varphi:S\to H$ be a continous bijective homomorphism between a connected Lie group $S$ and a locally arcwise-connected group $H$. Since a connected Lie group is locally compact and second countable, by {\cite[Theorem 4.1]{GP57}}, $\varphi$ is an isomorphism of topological groups.
\end{proof}
\end{tm}

To see that the pre-locally compactness is necessary, it suffices to look at any infinitely dimensional separable Banach space (such as $\mathbb{R}^{\omega}$).

We finish this chapter by examining groups that are continuous images of connected locally compact groups.

\begin{tm}
\label{tm:Charloc}
Let $H$ be an arbitrary topological group. Then the following conditions are equivalent:
\begin{enumerate}[label=(\roman*)]
\item $H$ is the image of a connected locally compact group by a continuous homomorphism.
\item There exist subgroups $L,K$ of $H$ where $L$ is a Lie group, $L\lhd H$, $K$ is compact and connected such that $H=KL$.
\item There exist subgroups $L,N$ of $H$ where $L$ is a virtual Lie group, $N$ is compact, $L,N\lhd H$ and $L,N$ commute pointwise such that $H=NL$ and $N\subseteq K$ where $K$ is some compact connected subgroup of $H$.
\end{enumerate}
\begin{proof}
The implication (iii)$\implies$(ii) is obvious. For (ii)$\implies$(i) let us denote by $\widetilde{L}$ the group $L$ with the topology of a connected Lie group. For $y\in K$ and $x\in L$ let $\varphi_y(x)=yxy^{-1}$. Observe that for $y\in K, \varphi_y\in Aut(L)$ and since $\widetilde{L}$ has the topology given by path components of open sets of $L$, $\varphi_y\in Aut(\widetilde{L})$. Now let us consider $\varphi:K\ni y\mapsto \varphi_y\in Aut(\widetilde{L})$. Note that $\widetilde{L}\rtimes_\varphi K$ is a locally compact topological space and $\widetilde{L}\rtimes_\varphi K\ni (x,y)\to x\cdot y\in H$ is a continuous surjective homomorphism. It is enough to show that $\widetilde{L}\rtimes_\varphi K$ is a topological group.  We recall that since $\widetilde{L}$ is a connected Lie group, $Aut(\widetilde{L})$ is locally compact and $\sigma$-compact. Now, since $K$ is compact and $Aut(\widetilde{L})$ is locally compact and $\sigma$-compact, after verifying that the graph of $\varphi$ is closed we get that $\varphi$ is continuous and thus $\widetilde{L}\rtimes_\varphi K$ is indeed a topological group. For the implication (i)$\implies$(iii) let $G$ be a locally compact connected group and let $\varphi:G\to H$ be a continuous surjective homomorphism. By Iwasawa's Theorem (\cite[Theorem 11]{Iwa49}, see also \cite[Corollary 13.20]{HM06}), there exist a connected Lie group $S$, a connected compact group $C$ and a continuous surjective homomorphism $\psi:S\times C\to G$. Let now $\theta:=\varphi\circ\psi$. Then $\theta:S\times C\to H$ is also a surjective continuous homomorphism. Then $L:=\theta(S)$ is a virtual Lie group and since $S\lhd S\times C$, $L\lhd H$. Now let $N=\theta(C)$. Then $N$ is a compact subgroup of $H$ and $N, L$ commute pointwise. Let now $M:=\psi(C)$. Then $M$ is a compact subgroup of a locally compact connected group $G$ and thus by \cite[Theorem 12.77]{HM07a} there exists a compact connected group $P$ such that $M\subseteq P\subseteq G$. Now if we put $K=\varphi(P)$ we get $N=\varphi(M)\subseteq\varphi(P)=K$ and $K$ is a compact connected subgroup of $G$.
\end{proof}
\end{tm}

As Example \ref{ex:XY} will show, $K$ in (ii) or (iii) need not be normal in $H$.

\section{Counterexamples}

In this section we provide important examples illustrating our characterisation of virtual Lie groups.

Let $\mathscr{A}=\{A \text{ subgroup of } \mathbb{R} \, : \, \overline{A}=\mathbb{R}\}$ and let $\mathbb{T}$ stand for the circle group. For $A\in\mathscr{A}$ let $G_A=\{\chi:A\to\mathbb{T} \text{ homomorphism}\}$ (equipped with the pointwise multiplication and the topology of pointwise convergence) be the Pontryagin dual (see {\cite[Section 9.5, p. 604]{AT08}}) of $A$ equipped with the discrete topology and let $H_A=\{A\ni a \mapsto e^{ati}\in \mathbb{T}\,|\,t\in\mathbb{R}\}$.

\begin{prop}
\label{prop:A}
For $A\in\mathscr{A}$:
\begin{enumerate}[label=(\arabic*)]
\item $G_A$ is a compact topological group and $H_A$ is its dense subgroup and thus a pre-locally compact group.
\item \label{item:virt} $H_A$ is a virtual Lie group.
\item\label{item:point} $H_A$ is NSS $\iff$ $\exists{a,b\in A\setminus\{0\}}:\frac{a}{b}\notin \mathbb{Q}$.
\item\label{item:metr} $H_A$ is non-metrisable $\iff$ $A$ is uncountable.
\end{enumerate}
\begin{proof}
That $G_A$ is a compact topological group follows from {\cite[Proposition 9.5.5]{AT08}} and that $H_A$ is dense in $G_A$ follows from {\cite[Corollary 9.6.1]{AT08}}. \ref{item:virt} is obvious and \ref{item:metr} follows from {\cite[Corollary 9.6.7]{AT08}}. We only need to show \ref{item:point}. We start by observing that $H_A$ not being NSS is equivalent to the following:
\begin{equation}\tag{*}
\label{charcont}
\parbox{0.9\textwidth}{For all finite $F\subseteq A$ and all $\varepsilon>0$ there exists $x>0$ such that $|e^{kaxi}-1|<\varepsilon$ for any $k\in \mathbb{Z}$ and $a\in F$.
}
\end{equation}

First let us assume that $H_A$ is not NSS. Let $\varepsilon<1/4$ and let $F=\{a,b\}$ for arbitrarily chosen $a,b\in A\setminus\{0\}$. Then by (*) there exists $x>0$ such that $\langle \{e^{axi},e^{bxi}\}\rangle$ is contained in $\mathbb{T}^+=\{z\in\mathbb{T}:Re(z)>0\}$. Then by the NSS property of $\mathbb{T}$, $e^{axi}=1$, so $ax\in 2\pi \mathbb{Z}$. Similarily, $bx\in 2\pi i \mathbb{Z}$, so $\frac{a}{b}\in\mathbb{Q}$. Now let us assume that any two nonzero elements of $A$ are commensurable. Then any finite $F\subseteq A$ is of the form $\{c\cdot\frac{p_1}{q_1},\ldots,c\cdot\frac{p_n}{q_n}\}$ for some $c>0$, $n\in\mathbb{N}$ and $p_i,q_i\in\mathbb{Z}$, $q_i>0$ for $i=1,\ldots,n$. If we put $q=lcm(q_1,\ldots,q_n)$ then $F=\{\frac{c}{q}\cdot k_1,\ldots, \frac{c}{q}\cdot k_n\}$ for some $k_i\in\mathbb{Z}$, $i=1,\ldots,n$. Now for $x=\frac{2q\pi}{c}$ we get $e^{kaxi}=1$ for all $a\in F$ and $k\in\mathbb{Z}$.
\end{proof}
\end{prop}

The above result lets us provide several interesting examples. The first one shows that the NSS condition is not necessary for being a virtual Lie group.

\begin{ex}
\label{ex:Q}
$H_\mathbb{Q}$ is an example of a virtual Lie group which is not NSS.
\end{ex}

Below is an example of a pre-locally compact non-metrisable NSS group, which is interesting since locally compact NSS groups are always metrisable.

\begin{ex}
\label{ex:R}
$H_\mathbb{R}$ is an example of a virtual Lie group which is NSS but is not metrisable.
\end{ex}


The following example illustrates why the condition of arcwise-connectivity is not sufficient for being a virtual Lie group.

\begin{ex}
\label{ex:omega}
$(H_\mathbb{Q})^\omega$ is an example of a pre-locally compact group that is arcwise connected but is not a virtual Lie group and it does not even admit a stronger topology which makes it a connected locally compact group.
\end{ex}

The above example is the only one that does not follow trivially from Proposition \ref{prop:A}, but we can see very quickly how it follows from it. If $(H_\mathbb{Q})^\omega$ was a virtual Lie group, then we could find a continuous bijective homomorphism $\varphi:S\to H$ where $S$ is some connected Lie group. Now we also have a continuous bijective homomorphism $\psi:\mathbb{R}^{\omega}\to H$ by the definition of $H_Q$. Now $\psi^{-1}\circ \varphi$ is a bijective homomorphism between two Polish groups so after veryfying that its graph is closed we get that it is a homeomorphism which leads to contradiction.

The last example shows that not every group that is a continuous homomorphic image of a conncected locally compact group is a pointwise product of two normal subgroups, one virtual Lie group and one compact connected group.

\begin{ex}
\label{ex:XY}
Let $X$ be the universal covering of $SL_2(\mathbb{R})$. Then $Z(X)\equiv\mathbb{Z}$ (cf. \cite[Proposition 17.2.3(3)]{HN12}) and thus there exists $x\in X$ such that $\langle x \rangle = Z(X)$. Let $Y$ be the group of 2--addic integers. Recall that such a group is compact, monothetic and totally disconnected. In particular, there exists $y\in Y$ such that $\overline{\langle y \rangle}=Y$. Let $u:\langle x \rangle\to \langle y \rangle$ be a homomorphism given by $u(x)=y$. Let $G=(X\times Y)/\Gamma(u)$. Then $G$ is a connected locally compact group that does not admit normal subgroups $K,L$ such that $L$ is a virtual Lie group, $K$ is a compact connected group and $G=KL$.
\end{ex}

We will end the article by proving the above statement, but first we introduce some auxiliary results.

\begin{lem}
\label{lem:figh}
Let $\varphi:G\to H$ be an open continuous surjective homomorphism between two locally compact groups. Then $\varphi^{-1}(H_0)=\overline{G_0\cdot \ker\varphi}$.
\end{lem}

\begin{lem}
\label{lem:XYG}
Let $X$ be a locally compact connected group and let $Y$ be a locally compact totally disconnected group. Let $X_1$ and $Y_1$ be subgroups of $X$ and $Y$ respectively and let $u:X_1\to Y_1$ be a continuous surjective homomorphism such that:
\begin{enumerate}
\item $[X,X_1]\subseteq \ker(u)$,
\item $Y_1 \subseteq Z(Y)$.
\end{enumerate}
Then the group $(X\times Y)/\Gamma(u)$ is connected if and only if $\overline{Im(u)}=Y$.
\end{lem}
\begin{proof}
First observe that both conditions (1) and (2) being satisfied is equivalent to $\Gamma(u)$ being a normal subgroup of $X\times Y$. Having established that, we want to determine when $((X\times Y)/\Gamma(u))_0=(X\times Y)/\Gamma(u)$. Let $\pi:(X\times Y)\to (X\times Y)/\Gamma(u)$ be the canonical projection. Then $((X\times Y)/\Gamma(u))_0=\pi(\pi^{-1}(((X\times Y)/\Gamma(u))_0))$. However by Lemma \ref{lem:figh}, $\pi^{-1}(((X\times Y)/\Gamma(u))_0)=\overline{(X\times Y)_0\cdot \ker\pi}$. Thus $((X\times Y)/\Gamma(u))_0=\pi(\overline{(X\times Y)_0\cdot \ker\pi})=\pi(\overline{(X\times\{e\})\cdot \Gamma(u)})=\pi(X\times \overline{Im(u)})$ which finishes the proof.
\end{proof}

The following lemma is the classical Open Mapping Theorem for locally compact $\sigma$-compact groups.

\begin{lem}
\label{lem:omt}
Let $G$ and $H$ be locally compact groups and let $G$ be $\sigma$-compact. Then every continuous surjective group homomorphism $\varphi\colon G\to H$ is open.
\end{lem}

We leave the following result as an easy exercise.

\begin{prop}
\label{lem:fixy}
Let $\varphi:X\to Y$ be an open continuous surjection where $X$ is locally compact and $Y$ is compact. Then there exists a compact subset $K$ of $X$ such that $\varphi(K)=Y$.
\end{prop}

We are now ready to justify Example \ref{ex:XY}.

\begin{prop}
All the assertions of Example \ref{ex:XY} are correct.
\end{prop}
\begin{proof}
Let us assume to the contrary that $K,L$ are normal subgroups of $G$ such that $K$ is compact connected, $L$ is a virtual Lie group and $G=KL$. Let $\pi:X\times Y\to G$ be the natural projection and let $H=\pi^{-1}(K)$. Then $H\lhd X\times Y$. Let $\xi=\pi|_H$. We observe that $H$ and $K$ are locally compact, $H$ is $\sigma$-compact and $\xi:H\to K$ is a continuous surjective homomorphism. By Open Mapping Theorem, $\xi$ is open. We now aim to prove

\begin{equation}\tag{$\bigstar$}
H_0\neq \{e\}.
\label{Honote}
\end{equation}

To that end let us assume that $H_0=\{e\}$. Then by the Van Dantzig Theorem (cf. \cite[Theorem 6.1.1]{Tao14}), $K=\{e\}$, which means that $G=L$, so $G$ is a Lie group. But, since $\pi|_Y$ is one-to-one $Y$ is embedded in $G$ by $\pi$ and thus $Y$ is also a Lie group. However, $Y$ is totally disconnected and compact, so $Y$ is finite, which leads to a contradiction. We now have shown (\ref{Honote}). Observe that, since $H\subseteq X\times Y$ and $Y$ is totally disconnected, $H_0\subseteq X\times \{e\}$. This means that $H_0=S\times \{e\}$ for some closed connected subgroup $S$ of $X$. Now $H_0$ is a characteristic subgroup of $H$ and $H$ is normal in $X\times Y$, so $H_0$ is normal in $X\times Y$. Thus $S\lhd X$ but $S$ is closed and connected and $X$ is a simple Lie group, so $S=\{e\}$ or $S=X$. Since the first option is not possible by (\ref{Honote}), $S=X$. By Lemma \ref{lem:figh}, $H=\xi^{-1}(K)=\overline{H_0\cdot \ker\xi}$. We further note that $\overline{H_0\cdot \ker\xi}=\overline{(S\times \{e\})\cdot \Gamma(u)}=\overline{(X\times \{e\})\cdot \Gamma(u)}=X\times\overline{Im(u)}=X\times Y$. Thus $\pi(X\times Y)=K$ which means that $G$ is compact. Now by Proposition \ref{lem:fixy} there exists a compact subset $F$ of $X$ such that $\pi(F\times Y)=G$. Now $\pi^{-1}(\pi(F\times Y))=X\times Y$, so $X\times Y=(F\times Y)\cdot \Gamma(u)=(F\cdot Z(X))\times Y$. Thus $F\cdot Z(X)=X$. This means that $X/Z(X)$ is compact, but $X/Z(X)\equiv SL_2(\mathbb{R})/Z(SL_2(\mathbb{R}))$ which leads to a contradiction.
\end{proof}

\bibliographystyle{amsalpha}
\bibliography{refs.bib}

\end{document}